%% file: 02_sparsification.tex
\documentclass[preprint,12pt]{elsarticle}

\journal{Computers and Mathematics with Applications}

\usepackage{subfigure}
\usepackage{graphicx}
\usepackage{amsmath}
\usepackage{amsthm}
\usepackage{amsfonts}
\usepackage{enumerate}
\usepackage[boxed,lined]{algorithm2e}
\usepackage{url}

\begin{document}

\begin{frontmatter}

%\singlespacing
%\setcounter{tocdepth}{2}

%%%%%%%%%%% title


\title{
Subspace-preserving sparsification of matrices with minimal
perturbation to the near null-space. Part II: Approximation and Implementation
}


\author{Chetan Jhurani}
\ead{chetan.jhurani@gmail.com}

\address{
	Tech-X Corporation\\
	5621 Arapahoe Ave\\
	Boulder, Colorado 80303, U.S.A.
}

%%%%%%%%%%% abstract

\begin{abstract}

This is the second of two papers to describe a matrix sparsification
algorithm that takes a general real or complex matrix as input and produces
a sparse output matrix of the same size.  The first paper~\cite{cite:ssa_1}
presented the original algorithm, its features, and theoretical results.

Since the output of this sparsification algorithm is a matrix rather than
a vector, it can be costly in memory and run-time if an implementation
does not exploit the structural properties of the algorithm and the matrix.
Here we show how to modify the original algorithm to increase its
efficiency. This is possible by computing an approximation to the exact result.
We introduce extra constraints that are automatically determined based on the input
matrix.  This addition reduces the number of unknown degrees of freedom but still
preserves many matrix subspaces.  We also describe our open-source library that
implements this sparsification algorithm and has interfaces in C++, C, and MATLAB.

\end{abstract}


\end{frontmatter}

%\tableofcontents
%\addtolength{\parskip}{\baselineskip}

\include{newcommand}

\section{Introduction}

In the first paper of this series~\cite{cite:ssa_1}, we presented a
matrix-valued linearly constrained convex quadratic optimization problem.
The application was to compute a sparsified matrix from a dense input
matrix such that the sparsification process led to small perturbations
in the lower end of the spectrum and preserved the null-spaces and
some other important properties of input matrix.

Our first objective here is to describe an enhancement that leads to an
approximate solution using less computational time.  The second objective
is to describe our open-source library, named TxSSA, that implements the sparsification
algorithm and is usable from C++, C, and MATLAB.  We also add to the theoretical
and numerical results presented earlier.

We briefly mention the original optimization problem.  Given a matrix
$A \in \Cmn$, we want to compute a sparse $X \in \Cmn$, that
solves the following problem.
\begin{align}
\label{eq:min_J}
\min_{X} & \half \norm{(X - A) \pinv{A}}^2_F + \half \norm{\pinv{A} (X - A)}^2_F \text{ such that } \nonumber\\
&\nullsp{A} \subseteq \nullsp{X}, \nonumber\\
&\nullsp{A^*} \subseteq \nullsp{X^*}, \text{ and } \nonumber\\
&X \text{ has a specified sparsity pattern. }
\end{align}
The symbol $\dagger$ is for Moore-Pensrose pseudoinverse, $*$ is for
Hermitian transpose, and $\mathcal{N}$ denotes the null-space of a matrix.
Typically the sparsity pattern is computed from $A$ but could be specified
separately.  In the first paper, we described an $L_p$ norm based algorithm
for sparsity pattern determination that works on individual rows and columns
to preserves entries with large magnitude.

Here is an outline of this paper. In Section 2, we describe some computational
issues with the optimization problem. In Section 3, we describe a solution to rectify
these issues, which is an algorithm to specify extra constraints in the optimization problem.
In Section 4, we deal with existence and uniqueness issues. In Section 5, we show
that specification of extra constraints still preserves useful matrix properties.
We give details of the numerical techniques we use in Section 6.  In Section 7,
we describe our implementation of the sparsification algorithm. Finally, in Section 8,
we show some new numerical results related to the enhancements we propose in this
paper.

\section{Computational issues with sparsification}

The sparsification problem discussed in~\cite{cite:ssa_1} and in~\Eq{eq:min_J} is mathematically simple because it
is a linearly constrained quadratic convex minimization problem.
The inputs are the Moore-Penrose pseudoinverse of input matrix $A$ and
the desired sparsity pattern $\pat{X}$.  We have algorithms to compute them.
However, the optimization problem is computationally expensive.  There are multiple reasons for this.
\begin{enumerate}
  \item The number of unknowns, which is equal to number of non-zeros in
  the sparsity pattern, can be significantly larger than $m$ or $n$, the
  input and output matrix sizes.
  \item The conditioning of the Hessian can be significantly worse compared
  to the conditioning of the input matrix.  This is evident by looking at the
  eigenvalue decomposition of the Kronecker sum Hessian, as shown
  in~\cite[Section 2.6]{cite:ssa_1}.
  %REF
  \item The graph corresponding to the Hessian does not necessarily have
  the structure required for sparse direct solvers to be efficient.
  Typically they lead to large fill-in even after fill-in reducing permutations
  are applied.
\end{enumerate}
See~\Sec{sec:numres} for a numerical example showing these phenomena.
The minimization problem is still solvable but the cost grows rapidly
with matrix size.  Our goal is to control the minimization cost.

Our solution, described ahead in detail, is to reduce the number of
unknowns by imposing simple linear equality constraints. 
The size of the reduced Hessian is then small enough to be
factored by a dense solver.  Although this changes the original problem and so we do
not compute the {\em exact} solution.  We will show in~\Sec{sec:numres} that the effect of such
an approximation is minor.

We call this process, using which we reduce the number of unknowns,
binning.  It is a mathematically well-defined process and requires an input parameter
for number of {\em bins}.
However, the results are not too sensitive once the parameter is sufficiently large.  By reducing system size, binning typically improves the conditioning
of the optimization problem so that is one less issue to worry about too.
In an earlier paper of ours~\cite{cite:cj_tma_bj_jcp}, we had not done any binning
and had to resort to iterative solvers and early termination to manage
the speed of our algorithm.  Similarly, in~\cite{AusBrez}, which was the
original work for such a sparsification algorithm, black-box dense matrix solvers were
used to invert all matrices. Simply speaking, the focus was on prototyping the features of the algorithm
rather than the solver.  These limitations are now removed as well.

\section{Binning the unknown entries}
\label{sec:binning}

The concept of binning the unknowns is simple. It is based on the numerical
evidence that when the {\em exact} minimization problem is solved, the
matrix locations that are nearly equal in value in the input matrix $A$
remain nearly equal in value in the output matrix $X$ (when present in the
sparsity pattern). This holds for entries belonging to
different rows and columns.  See~\cite[Section 6]{cite:ssa_1} for an example. %REF

We use this observation to enforce extra
linear equality constraints {\em a priori} between the entries of $X$ if
the corresponding entries in $A$ are nearly equal. Since these are very
simple constraints, it is easy to eliminate one unknown and reduce the
problem size by one (per constraint).  Of course, multiple entries of $A$
can be nearly equal in which case multiple corresponding unknown $X$ entries are
constrained to be in one {\em bin}.  Hence the name of the process.  We
typically reduce the problem size so that it contains a few hundred bins,
for arbitrary sparsity patterns.  Once the reduced Hessian is formed, such
a problem can be solved very quickly.

The main question is whether enforcing so many extra constraints via
binning will lead to a matrix $X$ which has suitable spectral properties
and is subspace-preserving.  This is a topic of other theoretical sections ahead and
some numerical results are presented in in~\Sec{sec:numres}.  It is
seen that the constraints due to binning conflict with the constraints due
to null-spaces in general (when the left or right null-spaces are
non-trivial).  We show how to avoid this issue and still get a suitable
matrix $X$.  First, we present a detailed description of the binning
algorithm ignoring its effects on optimization.

\subsection{A binning algorithm}

The input to the binning algorithm is the matrix $A \in \Rmn$ or $A \in
\Cmn$, the chosen sparsity pattern $\pat{A} \in \Rmn$, and the maximum number
of bins $N_{bin}$ per real or imaginary part.  If the matrices are complex,
the sparsity pattern is the same for both, real and imaginary, parts.
This is by design, see~\cite[Remark 3.10]{cite:ssa_1}.  But %REF
we keep $N_{bin}$ bins for the real part and $N_{bin}$ bins for the
imaginary part and they are binned separately.

The output of the binning algorithm
is one bin identifier corresponding to each non-zero in $\pat{A}$, per real
and imaginary part.  A bin identifier is a natural number.   Unknowns with the same bin identifier will be
constrained to be equal in the optimization process.  For complex matrices,
bin identifiers for real and imaginary parts don't overlap.  Essentially,
we choose the smallest bin identifier for the imaginary part to be one
greater than the largest bin identifier for the real part.

Let $\bin{A}$ denote the bin identifier matrix, which is of the same size as $A$.
It is real if $A$ is real and complex if $A$ is complex. All entries of
\bin{A} are non-negative integers.  Our convention is that $(\bin{A})_{ij} = 0$ if
$(\pat{A})_{ij} = 0$, in which case it is not a bin identifier.

Here is a basic binning algorithm in detail.
First we take all the entries of $A$ whose corresponding entry in $\pat{A}$ is one
and compute their minimum and maximum value.  This gives use the range.
The range is equally divided to form $N_{bin}$ bins and each entry
is then placed in a bin.  When all entries are binned, some bins may
remain empty and no identifier is assigned to them. The non-empty bins
are given an identifier sequentially.  All the matrix locations that
fall in the same bin will be constrained to be equal while optimizing.

As a simple example, a bin identifier can be computed by a simple formula
as follows.  Here $v$ is the real value for which an identifier is to
be computed, $\min$ is the minimum matrix value, $\max$ is the maximum
matrix value, and $h$ is $\frac{\max - \min}{N_{bin}}$.
$$
id = \lfloor (v - \min) h \rfloor + 1
$$
We compute such identifiers separately for the positive and negative ranges
in the input matrix.

Here is a concrete example using a $3\times 4$ real matrix that shows the process of
computing the sparsity pattern and bin identifiers.
\begin{verbatim}
    A =  5     4     1    -5
        -5     8    -7     7
         0     9    -7    -5

    A_pat = full(p_norm_sparsity_matrix(A, 0.6, 1));
         1     0     0     1
         1     1     1     1
         0     1     1     1

    A_id = full(bin_sparse_matrix(A, A_pat, 8))
         1     0     0     2
         2     3     2     3
         0     3     2     2
\end{verbatim}
For this example, the sparsity ratio is 0.6, $p$ for $L_p$ norm is 1, and
maximum number of bins is 8.  The function names and
arguments correspond to the MATLAB implementation (\Sec{sec:matlab}).
Note that $N_{bin}$ refers to maximum number of bins and the actual number
of bins can be less than that, three in this example.

The binning process can also be visually described using the MATLAB
command {\tt sortrows}.  We sort {\tt A} and {\tt A\_id} values together
after vectorization.
\begin{verbatim}
  B = sortrows([A(:) A_id(:)])'
  
   -7 -6 -5 -5 -5  0  1  4  5  7  8  9
    2  2  2  2  2  0  0  0  1  3  3  3
\end{verbatim}
This clearly shows that values close to each other are
assigned equal bins and that bins for values
not in the sparsity pattern are 0.

The binning algorithm is designed to satisfy some important
properties.  They are listed below without proofs. To simplify the discussion ahead, we define the notion of equivalence of bin identifier matrices.
The letters $i,j,k,$ and $l$ are indices within appropriate ranges.
\begin{definition}
Two real bin identifier matrices $\binsym^1$ and $\binsym^2$ are equivalent,
denoted by $\binsym^1 \sim \binsym^2$, if and only if they are of the same
size and $(\binsym^1)_{ij} = (\binsym^1)_{kl}$ whenever
$(\binsym^2)_{ij} = (\binsym^2)_{kl}$. For complex bin identifier matrices,
the equivalence holds if the real and imaginary parts are equivalent individually.
\end{definition}
We use the notion of equivalence rather than true equality because we
do not care about permutation of bin identifiers.  We only care about whether
they lead to the same subspace when all binning related constraints are imposed.

We now state a few important properties related to binning using conventions mentioned
above and the definition of equivalence.  The proofs are elementary and we skip them.
The matrix $A$ can be be real or complex.
\begin{enumerate}[{P}-1]
  \item $A_{ij} = 0 \implies (\bin{A})_{ij} = 0$.
  \item $\bin{\alpha A} \sim \bin{A}$ for $\alpha \in \reals \setminus \{0\}$. \label{prop:bin_homogeneous}
  \item $\bin{A^T} \sim (\bin{A})^T$.
  \item $\bin{A^*} \sim (\bin{A})^T$.
\end{enumerate}

\subsection{Interaction of binning and null-space constraints}

In general, one must be careful when introducing constraints when solving an
optimization problem.  The feasible region may become empty or trivial.
This can be the case when binning related equality constraints interact
with the null-space related equality constraints.  Here is an example.

Consider a general real square input matrix of size 1000 and rank 999.  Its
entries then satisfy 1000 equality constraints each due to left and right
null-spaces.  If we keep only 500 bins (and thus 500 unknowns) for the
output matrix, then the output cannot satisfy all the null-space
constraints unless all the entries are 0.  We have more homogeneous
equality constraints than the number of unknowns. Note that this is an
argument for a general situation.  One can create a special matrix and a
special sparsity pattern where the feasible region is non-trivial after binning.
In any case, even if the feasible region is non-trivial, we may be left with too
few degrees of freedom to have a minimization problem that really reduces the misfit we
are trying to minimize.

To avoid this issue but still allow binning to reduce computational
cost, we separate the imposition of binning constraint and null-space
constraints.  We solve two minimization problems whose details are provided
next.

\subsection{Optimization in two steps}

We first show the one-step {\em exact} optimization problem,
which was introduced in the first part in this series, for ease of comparison.
\begin{equation}
\begin{array}{l}
\displaystyle \min_X J(X;A) := \half \norm{(X - A) \pinv{A}}^2_F + \half \norm{\pinv{A} (X - A)}^2_F \text{ such that }\\[10pt]
\displaystyle \qquad \nullsp{A} \subseteq \nullsp{X}, \nullsp{A^*} \subseteq \nullsp{X^*}, \text{ and } X_{ij} = 0 \text{ if } (\pat{A})_{ij} = 0.
\end{array}
\end{equation}

Next we give the problem statements for the two optimization
problems supposed to use binning to approximate the output of the one-step
problem.  The first problem, computes a matrix $Y$ on which only sparsity and binning
constraints are imposed.  The second problem, uses $Y$ to compute
a matrix $X$ on which only sparsity and null-space
constraints are imposed.  The misfit functionals are different
on both.  Since we bin the real and imaginary parts separately,
we need to describe the optimization problem for real and complex
matrices slightly differently.

In the real case, $A,Y,X,\bin{A} \in \Rmn$, and we solve the following problem.
\begin{equation}
\label{eq:binreal}
\begin{array}{l}
\displaystyle \min_Y J(Y;A) := \half \norm{(Y - A) \pinv{A}}^2_F + \half \norm{\pinv{A} (Y - A)}^2_F \text{ such that }\\[10pt]
\displaystyle \qquad Y_{ij} = 0 \text{ if } (\pat{A})_{ij} = 0,\\[5pt]
\displaystyle \qquad Y_{ij} = Y_{kl} \text{ if } (\bin{A})_{ij} = (\bin{A})_{kl}.
\end{array}
\end{equation}

In the complex case, $A,Y,X,(\bin{A}) \in \Cmn$, and we solve the following problem.
\begin{equation}
\label{eq:bincomplex}
\begin{array}{l}
\displaystyle \min_Y J(Y;A) := \half \norm{(Y - A) \pinv{A}}^2_F + \half \norm{\pinv{A} (Y - A)}^2_F \text{ such that }\\[10pt]
\displaystyle \qquad Y_{ij} = 0 \text{ if } (\pat{A})_{ij} = 0,\\[5pt]
\displaystyle \qquad \re(Y_{ij}) = \re(Y_{kl}) \text{ if } \re((\bin{A})_{ij}) = \re((\bin{A})_{kl}), \text{ and }\\[5pt]
\displaystyle \qquad \im(Y_{ij}) = \im(Y_{kl}) \text{ if } \im((\bin{A})_{ij}) = \im((\bin{A})_{kl}).
\end{array}
\end{equation}
As is seen, only the binning constraints are different.  The main reason is that
binning is done on the real line, where values can be linearly ordered quite naturally,
and not in the complex plane.

The second problem imposes the null-space while maintaining the sparsity
imposed in the first problem.
\begin{equation}
\label{eq:nullimpose}
\begin{array}{l}
\displaystyle \min_X J_{null}(X;Y) := \half \norm{X - Y}^2_F \text{ such that }\\[10pt]
\displaystyle \qquad \nullsp{A} \subseteq \nullsp{X}, \nullsp{A^*} \subseteq \nullsp{X^*}, \text{ and } X_{ij} = 0 \text{ if } (\pat{A})_{ij} = 0.
\end{array}
\end{equation}
If we were to not use the Frobenius norm in the second
problem and use the original $\pinv{A}$ based misfit $J$ instead,
then the second problem would be equivalent to the non-binned problem
and expensive to solve.
Since the Hessian of the second problem is identity, it is much easier
to solve.

The main assumption behind this design is
that if the near null-space is `close' to the
null-space, then perturbing the near null-space little (in the first problem) would
lead to only a small perturbation to the null-space.
After
a matrix is found which does not satisfy the null-space
constraints, we just perturb the entries a little
so that it does satisfy them and the near null-space
is not perturbed a lot.  Of course, the assumption
of closeness of near and {\em actual} null-spaces is not true for all rank-deficient matrices.
But for problems that are discretizations of an infinite dimensional problem,
this is likely to be true.

\section{Existence and uniqueness in the two optimization steps}

We show that the two optimization problems always have solutions.  The
solution may not be unique for the first one when the input matrix
is rank-deficient in some atypical cases that don't arise in practice.
The second problem always has a unique solution.

\begin{lemma}
\label{lemma:bin_existence}
A minimizer always exists for the binning based minimization problems
posed in~\Eq{eq:binreal} and~\Eq{eq:bincomplex} for arbitrary imposed sparsity
and binning patterns.
\end{lemma}
\begin{proof}
The equality constraints for sparsity and binning are linear and
homogeneous.  Thus, the set of feasible solutions is non-empty.  For
example, the zero matrix is a feasible element.  Since the quadratic form
is convex for all $X$ (see Section 2.5.2 in~\cite{cite:ssa_1}), a minimizer %REF
always exists~\cite{cite:convex_opt}.
\end{proof}

We now prove that there is a globally unique minimizer for a full-rank
input matrix.
\begin{lemma}
\label{lemma:bin_unique}
If $A$ is full-rank, the binning based minimization problems posed 
in~\Eq{eq:binreal} and~\Eq{eq:bincomplex} have
globally unique minimizers for arbitrary imposed sparsity and binning
patterns.
\end{lemma}
\begin{proof}
As mentioned in Section 2.5.2 in~\cite{cite:ssa_1}, the quadratic form is %REF
strictly convex on $\Cmn$ if $A$ is full-rank.  In particular, it is
strictly convex on the subspace of those matrices that satisfy the sparsity
and binning constraints.  Since the feasible set is convex and non-empty,
the minimizer is globally unique.
\end{proof}

The solution may not be unique if the following three disparate conditions hold --
the input matrix $A$ is rank-deficient, the number of bins is too large,
and the number of sparsity constraints is too few. In such cases, the
Hessian corresponding to the quadratic form is rank-deficient on the space
of all matrices in $\Cmn$ and the constraints may not be sufficient to
remove this deficiency for the reduced Hessian.  This is a case of little
practical interest because we will typically have sufficient sparsity and
the number of bins will be finite and not too large.  Both these situations
make the reduced Hessian full-rank leading to a globally unique minimizer.

We show that second optimization problem with null-space constraints and a
simpler objective functional has a solution that exists and is unique.

\begin{lemma}
\label{lemma:null_unique_exist}
A globally unique minimizer exists for the null-space imposing minimization problem posed
in~\Eq{eq:nullimpose} for arbitrary imposed sparsity pattern and null-space
constraints.
\end{lemma}
\begin{proof}
The proof follows from the facts that the equality constraints related to
sparsity and null-spaces are linear and homogeneous and that the quadratic form is strictly convex.
\end{proof}

\section{Effect of binning on preserving structure}

As we have stressed, one of our goals is preserving common input matrix
structures, if it has any.  For example, being Hermitian, skew-Hermitian,
etc.  We had shown in the first paper~\cite{cite:ssa_1} that the {\em
exact} optimization problem preserves many such structures.  Here we intend to
show that we have not completely destroyed this property by binning and by
approximating the original problems by splitting into two separate
problems.

When binning is performed, our numerical evidence shows that the output sparse matrices automatically belong
to certain subspaces if the input matrix belongs to them.  In particular,
this property holds for Hermitian, complex-symmetric, 
circulant, centrosymmetric, and persymmetric matrices and also for each of
the skew counterparts except skew-circulant.
This holds for the intermediate output
matrix ($Y$, result of the first problem) and for the final output matrix
($X$, result of the second problem).  This is weaker than the result
when binning is not performed (see ~\cite[Section 5]{cite:ssa_1} because %REF
Hamiltonian, skew-Hamiltonian,
and skew-circulant matrices are absent.  We believe that a slightly modified
binning algorithm, where binning constraints like $Y_{ij} = - Y_{kl}$ are
used (compare~\Eq{eq:binreal}) might lead to a stronger result where these
three structures are also preserved.  However, this is just a conjecture as of now.

We prove below the result that binning based minimization
preserves the Hermitian, skew-Hermitian, and complex-symmetric subspaces.
We focus on these three only because these are common subspaces in
Finite Element practice.

As defined in the previous paper, let $\h{Y} = Y$ or $Y^*$ or $Y^T$,
where $op$ is a placeholder and means operation.  Define a multiple-choice
operation
$g$,
$$
g(Y) = \alpha \h{Y}
$$
where $\alpha$ is 1 or $-1$.
The conjugate transpose is important for complex matrices only.

\subsection{Structure modification when binning only}

We now show that in the first optimization problem, when binning is
performed, the output matrix is transformed in a predictable fashion
if the input matrix is transformed in a certain way.  First we need
to show how the bin identifiers transform.

\begin{lemma}
\label{lem:bin_invariance}
Let $A$ be a real matrix.
If $g(Y) = \alpha Y$, $(\bin{A})_{ij} = (\bin{A})_{kl}$ implies
$(\bin{g(A)})_{ij} = (\bin{g(A)})_{kl}$.
If $g(Y) = \alpha Y^T$, $(\bin{A})_{ij} = (\bin{A})_{kl}$ implies
$(\bin{g(A)})_{ji} = (\bin{g(A)})_{lk}$.
\end{lemma}
\begin{proof}
We give an outline of the proof.  The statement follows from the formula of bin identifier.
The bin identifier corresponding to a matrix location depends
only on the matrix entry value at that location and other globally common values $-$
number of bins, the maximum and minimum values in the positive
and negative range of matrix values.  Hence equality of
bin identifiers before transformation leads to equality of bin identifiers after
transformation of the input matrix under sign flip and transpose.
\end{proof}
We skip the proof for the complex matrix $A$ since it is based
on similar ideas.

We come to the main result concerning the first optimization problem.
\begin{theorem}
\label{thm:bin_invariance}
If $X$ solves the binning based minimization problems
posed in~\Eq{eq:binreal} (real case) and~\Eq{eq:bincomplex} (complex case) for a
given input $A$, then $g(X)$ solves it for input $g(A)$.
\end{theorem}
\begin{proof}
The proof is immediate based on the observations that the sparsity and binning
constraints as well as the misfit transform in a predictable fashion when the inputs are transformed under $g$.
Lemma 5.2 in~\cite{cite:ssa_1} states the proof for invariance of misfit. %REF
Lemma~\ref{lem:bin_invariance} gives the proof for predictable variation of binning
constraints.  The situation is similar to the one in proof of \cite[Theorem 5.5]{cite:ssa_1} %REF
and the result follows.
\end{proof}

\subsection{Structure modification when null-space is imposed}

We now show that in the second optimization problem, when null-space is
imposed, the output matrix is transformed in a predictable fashion
if the input matrix is transformed in a certain way.  First we need
to show that the misfit functional transforms predictably.

\begin{lemma}
\label{lem:J_invariance}
$J_{null}(X;Y) = J_{null}(g(X);g(Y))$, where $J_{null}(X;Y)$ is the misfit functional defined in
\Eq{eq:nullimpose}.
\end{lemma}
The proof is based on elementary linear algebra identities and we skip it.

We can now prove the main result concerning the second optimization problem.
\begin{theorem}
\label{thm:null_invariance}
If $X$ solves~\Eq{eq:nullimpose} for given inputs $A$ and $Y$, then $g(X)$ solves it
for inputs $g(A)$ and $g(Y)$.
\end{theorem}
\begin{proof}
The proof is similar to the proof of Theorem~\ref{thm:bin_invariance} above.
The relevant ingredients come from Lemma~\ref{lem:bin_invariance} and
Theorem 5.5 in~\cite{cite:ssa_1}. %REF
\end{proof}

\subsection{Preserving specific structures}

We can now state and prove the main result.  It shows that the current version of binning, as
describe in~\Sec{sec:binning}, preserves the most common matrix subspaces.

\begin{theorem}
If the input matrix $A$ for the problems in~\Eq{eq:binreal}, ~\Eq{eq:bincomplex},
and~\Eq{eq:nullimpose} belongs to one of the
following matrix subspaces -- Hermitian, complex-symmetric, skew-Hermitian, or
skew-complex-symmetric -- then the output matrix $X$ and intermediate
matrix $Y$ belong to the same subspace.
\end{theorem}
\begin{proof}
The proof is similar to the one give for Theorem 5.5 in~\cite{cite:ssa_1}. %REF
The difference is that the operator that preserves the structure of
constraints and misfit is given by $g = g(Y) = Y$ or $Y^T$ or $Y^*$ or $-Y^T$ or $-Y^*$
and involves no permutation choices.

If $A$ is Hermitian, or complex-symmetric, or skew-Hermitian, or
skew-complex-symmetric, it implies that $g(A) = A$.  Since the intermediate
and output matrices for $g(A)$ would be $g(Y)$ and $g(X)$, and the solution
is unique, we have $Y = g(Y)$ and $X = g(X)$.
Thus $X$ and $Y$ preserve the relevant property satisfied by $A$.
\end{proof}

\section{Computational choices}

We now discuss in some detail the choices we have made in computing
the Moore-Penrose pseudoinverse of the input matrix, the null-spaces,
and solving the two optimization problems described earlier.

\subsection{Computing the pseudoinverse}

For an arbitrary dense input matrix, the Singular Value Decomposition (SVD)
is the most reliable algorithm to compute the pseudoinverse.  However, we
expect our input matrices to be not ill-conditioned.  Otherwise sufficient
sparsification and preservation of near null-space are contradictory
requirements.  Thus, we use a cheaper algorithm, like rank-revealing QR
factorization, for general matrices.  If the matrix has special properties,
for example, being Hermitian positive-definite, then pivoted Cholesky can be
used and it will be faster in general but we have not used it in our
results.

For $A \in \Cmn$, a rank-revealing QR factorization is
$$
A = Q R P^T,
$$
where $Q \in \Cmm$ is unitary, $R \in \Cmn$ is upper triangular, and $P \in
\Rnn$ is a permutation matrix.  Matrices $Q$ and $R$ can be partitioned as
follows.  Let $A$ have rank $r$.
$$
Q =
\begin{bmatrix}
 Q_{1} & Q_{2}
\end{bmatrix}
$$
$$
R = 
\begin{bmatrix}
 R_{11} & R_{12}  \\
      0 &     0
\end{bmatrix}
$$
Here $Q_1$ is $m \times r$, $Q_2$ is $m \times (m-r)$, $R_{11}$ is
$r \times r$, and $R_{12}$ is $r \times (n-r)$.  Note that numerically
the bottom right block of $R$ will not be exactly zero-valued in general.  We
can change it to zero, like we have above, by computing the numerical rank.

The pseudoinverse of $A$ can be expressed as
$$
\pinv{A} = P R^* Q^*.
$$
This requires the pseudoinverse of the upper triangular $R$, which can be
computed easily~\cite{cite:bjorck,cite:cjldp2}. The rank-revealing QR
factorization is implemented in LAPACK as {\tt xGEQP3} and uses level-3
BLAS functions~\cite{cite:geqp3}.

\subsection{Computing left and right null-space bases}

The null-space bases are a by-product of the algorithm used to compute the
pseudoinverse.  In case of the rank-revealing QR factorization, the right
$m-r$ columns of $Q^*$ form an orthonormal basis of the left null-space.
One could similarly compute the QR factorization for $A^*$ to obtain a
basis for the right null-space.  However, we can use the information
provided by the QR factorization of $A$ itself to compute the right
null-space and avoid a second factorization.

It is seen that
$$
A \left( P
\begin{bmatrix}
 -R_{11}^{-1} R_{12}\\
 I_{n-r}
\end{bmatrix} \right)
=
Q 
\begin{bmatrix}
 R_{11} & R_{12}  \\
      0 &     0
\end{bmatrix}
 P^T  P
\begin{bmatrix}
 -R_{11}^{-1} R_{12}\\
 I_{n-r}
\end{bmatrix} 
= 0.
$$
This means
$$
P
\begin{bmatrix}
 -R_{11}^{-1} R_{12}\\
 I_{n-r}
\end{bmatrix}
$$
is a basis for the right null-space.  Unlike the left null-space basis
computed above, it is not orthonormal.  However, its columns can be orthonormalized.
Since we expect the rank deficiency to be a small constant, this is quite inexpensive.

\subsection{Solving the binning based optimization problem}

The reduced Hessian of the binning based optimization (see~\Sec{sec:binning})
can be easily computed from the full Hessian because the equality
constraints are very simple.  The size of reduced Hessian is
equal to the actual number of bins, a number typically less than
1000, and it is generally not too sparse (see~\Sec{sec:numres}).  Hence, 
Cholesky factorization is the natural algorithm for solving
the reduced optimization problem.

\subsection{An iterative method to impose the null-space}

The second optimization problem (\Eq{eq:nullimpose}) is a linearly constrained
convex quadratic problem.  It needs to be solved only if
the input matrix has left or right null-spaces.

The quadratic term is as simple mathematically
as it can be.  Its second derivative is the identity operator between
matrix domain and range spaces.  We have some
flexibility in choosing the basis for null-space but the natural choice
that makes the problem reasonably well-conditioned is choosing an
orthonormal basis.  We expect the left and right nullities to be small
constants.  Thus, if we have a non-orthonormal basis, making it
orthonormal is a cheap operation and is beneficial in general.  Finally, to
solve the minimization problem, we use the Uzawa algorithm with conjugate
directions~\cite{cite:braess}.  Since the Hessian of quadratic form is
identity, a matrix need not be inverted in each iteration.  Because of such
well conditioning, one can achieve the machine epsilon relative accuracy in
very few iterations.

\section{TxSSA: A sparsification library}

We describe version 1.0 of TxSSA, which is a library implemented in C++ and
C with interfaces in C++, C, and MATLAB. It implements all the mathematical
algorithms described in this paper and the previous part~\cite{cite:ssa_1}.
TxSSA is an acronym for Tech-X Corporation Sparse Spectral Approximation.
The code is distributed under the BSD 3-clause open-source license and is
available here.
\begin{center}
\url{http://code.google.com/p/txssa/}
\end{center}
The library works on multiple operating systems and can be built wherever
CMake~\cite{cite:cmake} is available.  The {\em exact} but expensive
minimization problem described in the previous part is implemented only in
MATLAB and not available in the C++ and C library.  The reason for this is
that the exact problem is more of an experimental feature and not useful
for practical applications.  A couple of programs using the library are
provided in a {\tt examples} sub-directory of the library distribution.

We now describe the most important subset of the full available interface.
We do not discuss the functions related to complex types and ignore some
other functions as well.
The C++ and C library interface is made available in file {\tt txssa.h}
in the {\tt include} directory.  The interface consists of (almost) pure functions
and minimal number of C++ classes and C structs.  A `pure' function is similar to a mathematical function in that
it is deterministic and its output depends only on its input arguments and has no
side-effects like changing shared resources.  We use the term `almost' because error
logging does lead to side-effects in case of erroneous behavior as does
heap memory allocation.

\subsection{The C++ implementation and interface}

We use C++ templates based on four template parameters in the library
interface as well as implementation.  We use explicit template instantiation for
common types so it does not increase compilation time for library users.
\begin{itemize}
  \item {\tt index\_type:}  This is used for matrix sizes.  Typical values
  are {\tt int}, {\tt unsigned int}, {\tt long long}, etc.
  \item {\tt offset\_type:} This is used for offsets in compressed sparse
  row (CSR) matrix format. Typical values are {\tt int}, {\tt unsigned int},
  {\tt long long}, etc.  The only condition is that {\tt
  sizeof(offset\_type)} $\ge$ sizeof(index\_type).

  \item {\tt value\_type:}  This is used for floating point values for
  functions related to real matrices. Typical values are {\tt double} and
  {\tt float}.

  \item {\tt scalar\_type:}  This is used for floating point values for
  functions related to complex matrices. Typical values are {\tt double}
  and {\tt float}.
\end{itemize}
The C based interface can be used for concrete data types.

If the input matrix structure is known a priori or not known, it can be
specified using an {\tt enum}.  The names are self-explanatory.  Currently
the code does not fully utilize all the optimizations possible if
matrix type is known a priori.
The input matrix must be stored in column-oriented order.
\begin{verbatim}
    enum ssa_matrix_type
    {
        ssa_matrix_type_undefined = -1,
        ssa_matrix_type_general,
        ssa_matrix_type_hermitian_pos_def,
        ssa_matrix_type_hermitian_pos_semi_def,
        ssa_matrix_type_hermitian,
        ssa_matrix_type_skew_hermitian,
        ssa_matrix_type_complex_symmetric,
        ssa_matrix_type_num_types
    };
\end{verbatim}
One caveat is that the full input matrix has to be provided even if a matrix has
a special structure.  This is unlike the BLAS and LAPACK functions where,
for example, only one half of a symmetric matrix needs to be given.

Here are the arguments for {\tt ssa\_lpn}, a function to sparsify a real
matrix based on $L_p$ norm sparsity pattern and two step optimization
process.
\begin{verbatim}
int ssa_lpn(
    index_type            num_rows,           // >= 1
    index_type            num_cols,           // >= 1
    const value_type*     col_values,         // != 0
    index_type            col_leading_dim,    // >= num_rows
    value_type            sparsity_ratio,     // in [0,1]
    value_type            sparsity_norm_p,    // in [0,inf]
    offset_type           max_num_bins,       // >= 0
    bool                  impose_null_spaces, // true/false
    enum ssa_matrix_type  matrix_type,
    ssa_csr<index_type, offset_type, value_type>& out_matrix);
\end{verbatim}
Here {\tt ssa\_csr} is data type that has the following three members to
specify the output CSR data.
\begin{verbatim}
    offset_type* row_offsets;
    index_type*  column_ids;
    value_type*  values;
\end{verbatim}
The return value is zero on a successful completion.  The {\tt
col\_leading\_dim} argument is the leading dimension just like it is in
BLAS and LAPACK. The {\tt sparsity\_ratio} argument is the $q$ parameter in
the algorithm to compute sparsity pattern and {\tt sparsity\_norm\_p} is
the $p$ in $L_p$ norm.  The value for {\tt inf} is specified by
\begin{center}
{\tt std::numeric\_limits<value\_type>::infinity()}
\end{center}
and is provided by the C++ {\tt <limits>} header.  {\tt max\_num\_bins} is
the number bins.  If it is 0, then each unknown entry is given a separate
bin.  Using a value like 1000 is good enough to obtain the benefits of
binning.

\subsection{The MATLAB implementation and interface}
\label{sec:matlab}
The interface to our MATLAB implementation consists of the following four
functions.
\begin{verbatim}
    X = ssa_compute_exact(A, ratio, p)
    X = ssa_compute_exact_for_pat(A, A_pat)
    X = ssa_compute(A, ratio, p, max_num_bins,
                    impose_null_spaces)
    X = ssa_compute_for_pat(A, A_pat, max_num_bins,
                            impose_null_spaces)
\end{verbatim}
The output {\tt X} is a sparse matrix and the arguments have the following
meaning.
\begin{itemize}
  \item {\tt A} is an input real or complex matrix
  \item {\tt max\_num\_bins} is a non-negative integer (1000 is a reasonable choice)
  \item {\tt A\_pat} is a pattern matrix of same size as A with zeros and ones
  \item {\tt p} is a power p in $[0,\infty]$ (1 is good enough typically)
  \item {\tt ratio} is a sparsity ratio in $[0,1]$ (0 means more sparse, 1 means less)
  \item {\tt impose\_null\_spaces} an option to impose null-spaces (true or false)
\end{itemize}
The functions with {\tt exact} word in the name solve the {\em exact}
minimization problem.  The functions with {\tt for\_pat} words in the name
solve the minimization problem for a given pattern matrix rather than the
one computed from $L_p$ norm based algorithm.

\subsection{The C interface}
The C interface to the library is similar to the C++ interface. The
differences are due to lack of templates in C and the lack of automatic
destructor calls.  We show the functions for real double precision
matrices.  Hence the letter {\tt `d'} in the function names below.  For all
functions, we fix {\tt index\_type} and {\tt offset\_type} to be {\tt int}.

\begin{verbatim}
int ssa_d_lpn(
    int                  num_rows,           // >= 1
    int                  num_cols,           // >= 1
    const double*        col_values,         // != 0
    int                  col_leading_dim,    // >= num_rows
    double               sparsity_ratio,     // in [0,1]
    double               sparsity_norm_p,    // in [0,inf]
    int                  max_num_bins,       // >= 0
    int                  impose_null_spaces, // true/false
    enum ssa_matrix_type matrix_type,
    struct ssa_d_csr*    out_matrix);
\end{verbatim}
The return value is zero on a successful completion.  The data type {\tt
ssa\_d\_csr} is to store the sparse matrix.  Since its
destructor will not be called automatically, one needs to use the
following deallocation function after the matrix is created (and used).
\begin{verbatim}
void ssa_d_csr_deallocate(struct ssa_d_csr* out_matrix);
\end{verbatim}

\section{Numerical results}
\label{sec:numres}

We now present some binning based sparsification results for a fixed real asymmetric matrix
$A \in \reals^{40\times 40},$ where
\begin{equation}
\label{eq:A_sample}
A_{ij} = \cos(3^{\frac{1}{4}} i^{\frac{1}{2}} j)^5,
\end{equation}
and the indices start at 1.  This is the matrix that was used
in the first part~\cite{cite:ssa_1}.  Out objective here
is to show that binning leads to a useful sparsified matrix
$X$ such that $\pinv{A}X$ is well-conditioned.

To motivate the importance of binning, first we show the sparsity patterns of the Hessian in the
two optimization problems -- when no binning is performed
and when binning is performed.  See~\Fig{fig:bin_no_bin}.
These are for sparsity parameters $p = 1$ and $q = 0.8$.  The
un-binned problem has 597 unknowns and the binned problem has 321
unknowns.  The condition number of the un-binned Hessian is $1.15 \times 10^5$,
which is much higher than the condition number of $A$, which is 621.
Another issue is that the structure of the Hessian is not amenable
to sparse direct factorization.  This is because of high fill-in
even if standard fill-in reducing permutations are used. \Fig{fig:no_bin_L}
shows the lower triangular factor the un-binned Hessian computed in MATLAB.

\begin{figure}
\begin{center}
	\subfigure[Un-binned Hessian $597 \times 597$]
	{
  \includegraphics[scale=0.5]{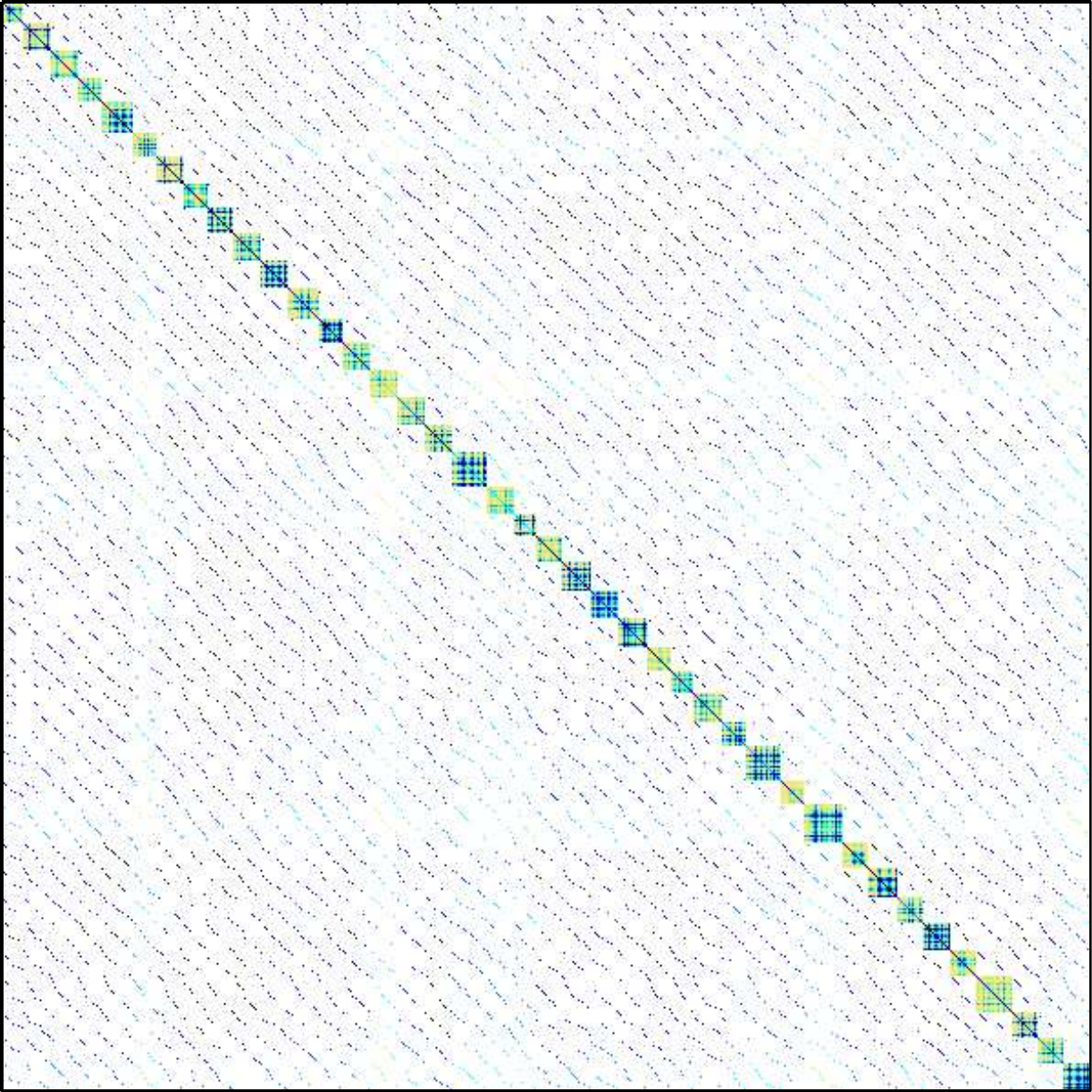}
	}
	\subfigure[Binned Hessian $321 \times 321$]
	{
	\includegraphics[scale=0.5]{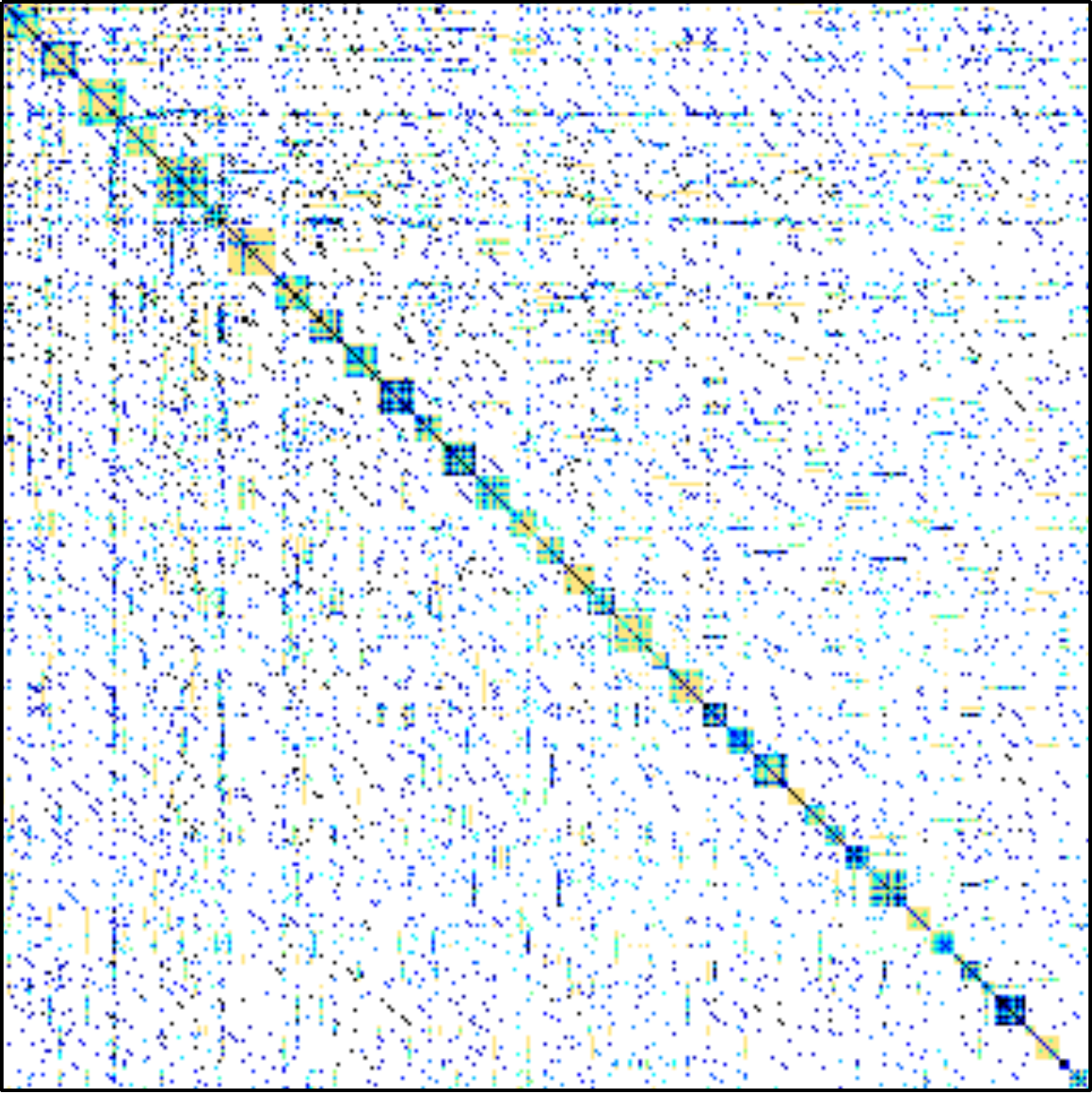}
	}
\caption{Two Hessians corresponding to the matrix specified in~\Eq{eq:A_sample}. Plots
generated using the {\tt cspy} program~\cite{Suitesparse} where darker pixel values correspond to larger magnitude.}
\label{fig:bin_no_bin}
\end{center}
\end{figure}

\begin{figure}
\centering
\includegraphics[scale=0.5]{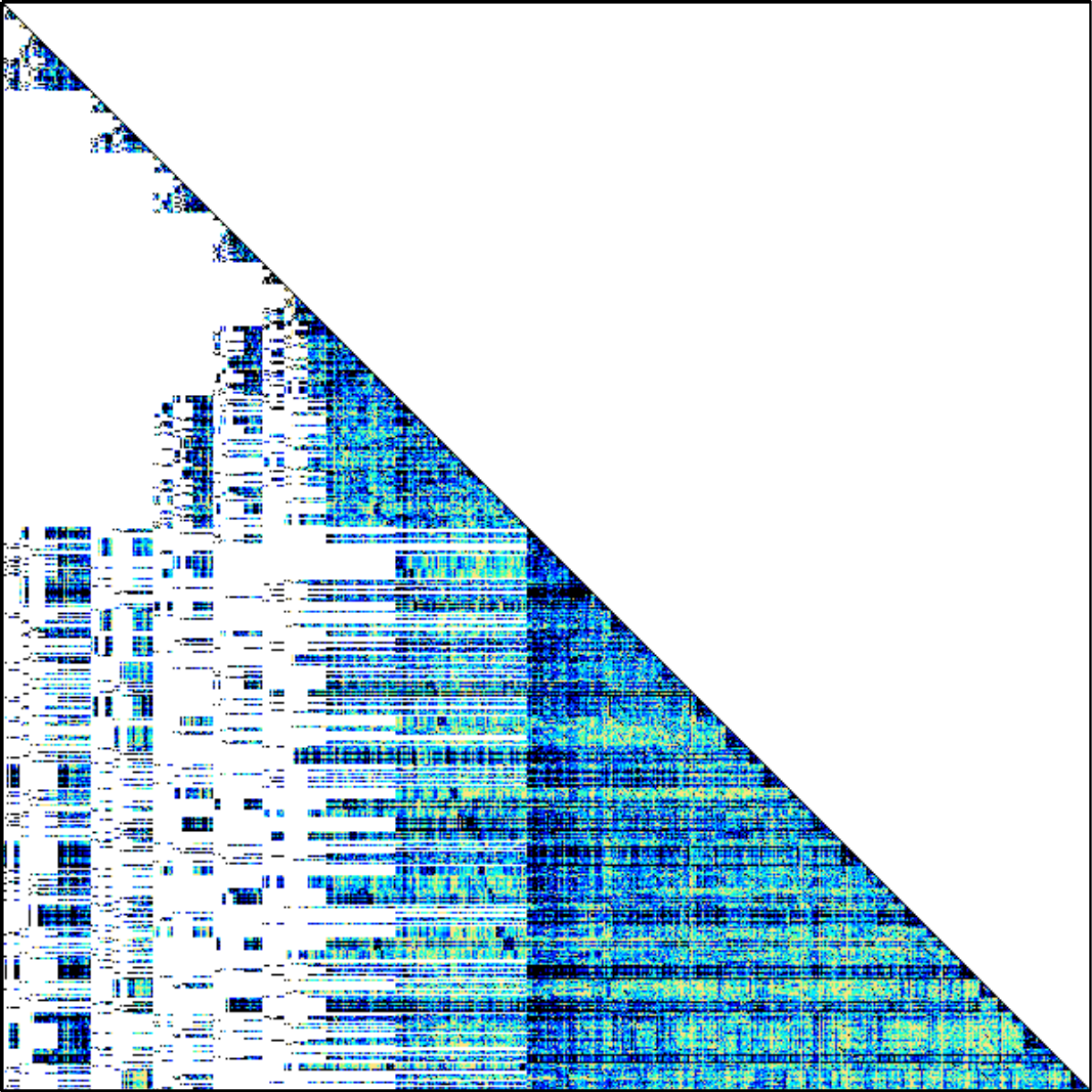}
\caption{The lower triangular factor of the Hessian matrix shown in \Fig{fig:bin_no_bin}(a) after the matrix
is reordered by MATLAB to reduce fill-in.  High fill-in mean sparse direct solvers will not be efficient.}
\label{fig:no_bin_L}
\end{figure} 

We now show that once $N_{bin}$, the number of bins, is
sufficiently large, the effects of binning on spectral properties of $X$
are little. \Fig{fig:nbin} shows the condition numbers of
$\pinv{A} X$ and relative Frobenius norm differences $\norm{\pinv{X} - \pinv{A}}_F / \norm{\pinv{A}}_F $
for various number of bins.  The lowest condition
number is approximately 8.37, which is very close to
the `exact' answer 4.73 obtained without binning in~\cite{cite:ssa_1}.

\begin{figure}
\begin{center}
	\subfigure[$\text{cond}(\pinv{A} X)$]
	{
    \includegraphics[scale=0.4]{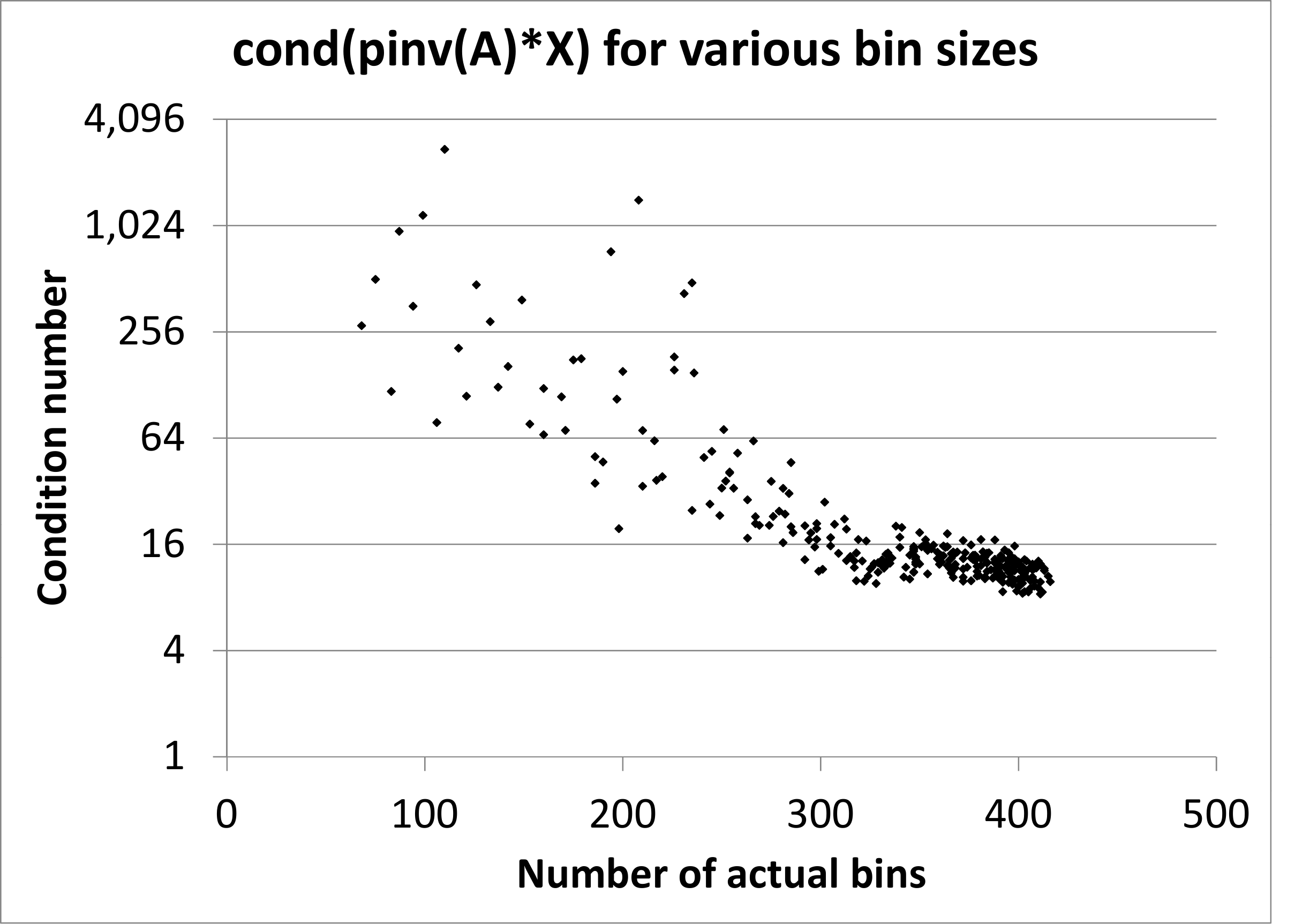}
	}
	\subfigure[$\norm{\pinv{X} - \pinv{A}}_F / \norm{\pinv{A}}_F $]
	{
    \includegraphics[scale=0.4]{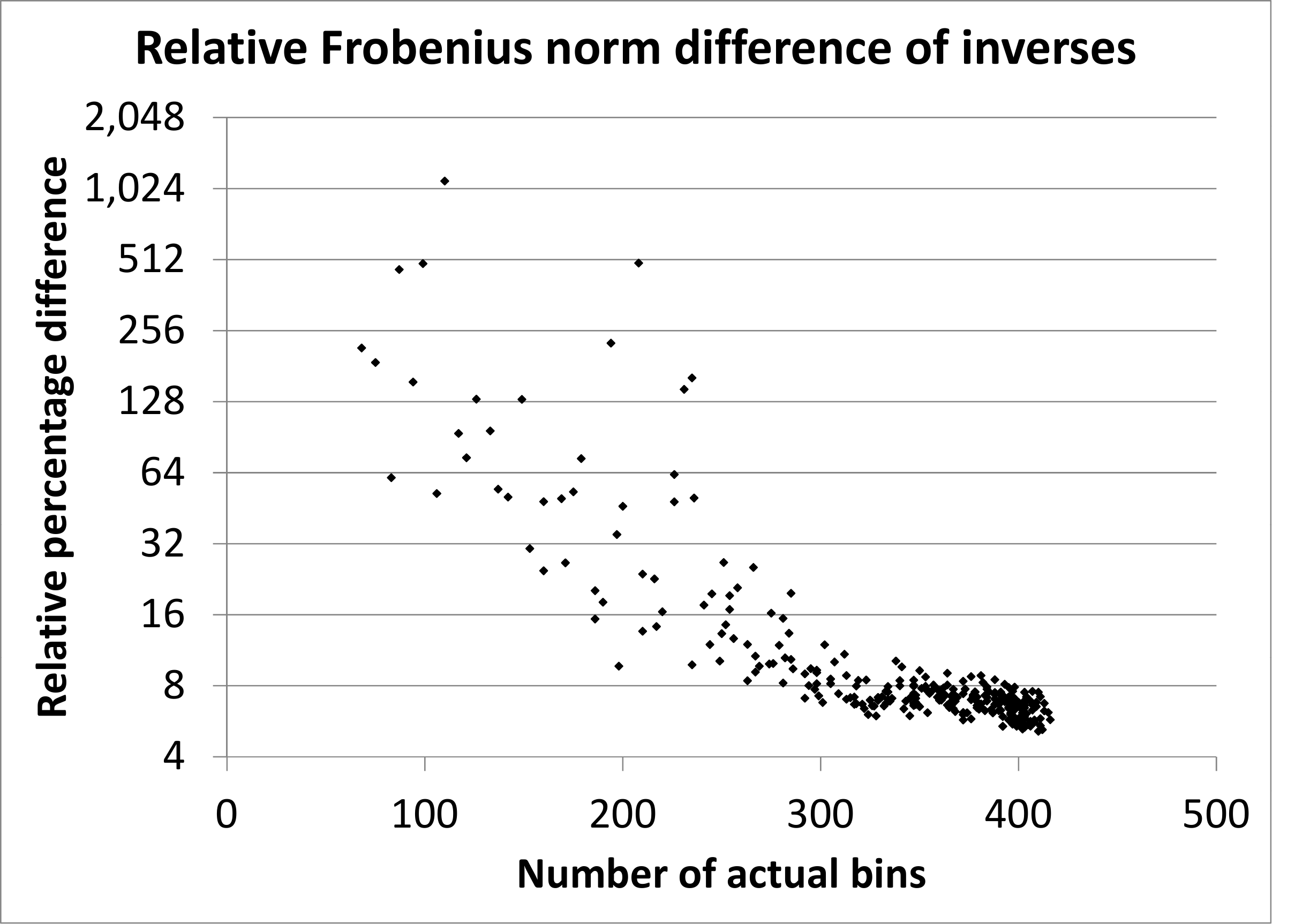}
	}
\caption{The figures shows that the spectral properties of of sparsified $X$ as well as the difference of inverses
are not too perturbed even if binning is used with sufficient number of bins.  The large differences
for low number of bins is due to a few small singular values that are not approximated well.}
\label{fig:nbin}
\end{center}
\end{figure}

\vspace{1cm}

\noindent \textbf{ \large{Acknowledgements}}

This work was partially supported by the US Department of Energy SBIR Grant
DE-FG02-08ER85154.  The author thanks Travis M. Austin, Marian Brezina,
Leszek Demkowicz, Ben Jamroz, Thomas A. Manteuffel, and John Ruge for many
discussions.

\bibliographystyle{elsarticle-num}
\bibliography{02_sparsification}

\end{document}

%% file: newcommand.tex
%-------------------------------------------------------------------------------

\newcommand{\matspace}[3]{\ensuremath{\mathbb{#1}^{#2 \times #3}} }

\newcommand{\Rnn}  {\matspace{\mathbb{R}}{n}{n}}
\newcommand{\Rmn}  {\matspace{\mathbb{R}}{m}{n}}
\newcommand{\Rmm}  {\matspace{\mathbb{R}}{m}{m}}

\newcommand{\Cnn}  {\matspace{\mathbb{C}}{n}{n}}
\newcommand{\Cmn}  {\matspace{\mathbb{C}}{m}{n}}
\newcommand{\Cmm}  {\matspace{\mathbb{C}}{m}{m}}
\newcommand{\Crn}  {\matspace{\mathbb{C}}{r}{n}}
\newcommand{\Crr}  {\matspace{\mathbb{C}}{r}{r}}

%-------------------------------------------------------------------------------

\newcommand{\vecspace}[2]{\ensuremath{\mathbb{#1}^{#2}} }

\newcommand{\Rn}   {\vecspace{\mathbb{R}}{n}}
\newcommand{\Cn}   {\vecspace{\mathbb{C}}{n}}

\newcommand{\Rm}   {\vecspace{\mathbb{R}}{m}}
\newcommand{\Cm}   {\vecspace{\mathbb{C}}{m}}

\newcommand{\reals} {\ensuremath{\mathbb{R}}}
\newcommand{\complex} {\ensuremath{\mathbb{C}}}

\newcommand{\re} {\ensuremath{\operatorname{Re}}}
\newcommand{\im} {\ensuremath{\operatorname{Im}}}

%-------------------------------------------------------------------------------

\newcommand{\norm}[1]{\ensuremath{\left| \left| #1 \right| \right|} }

\newcommand{\abs}[1]{\ensuremath{\left| #1 \right|} }

\newcommand{\pinv}[1]{\ensuremath{{#1}^{\dagger}} }
\newcommand{\inv}[1]{\ensuremath{{#1}^{-1}} }
\newcommand{\h}[1]{\ensuremath{{#1}^{op}} }

\newcommand{\nullsp}[1]{\ensuremath{\mathcal{N}(#1)} }
\newcommand{\rangsp}[1]{\ensuremath{\mathcal{R}(#1)} }

\newcommand{\patsym}[0]{\ensuremath{\mathcal{Z}} }
\newcommand{\pat}[1]{\ensuremath{\mathcal{Z}(#1)} }

\newcommand{\binsym}[0]{\ensuremath{\mathcal{B}} }
\newcommand{\bin}[1]{\ensuremath{\mathcal{B}(#1)} }

\newcommand{\lagr}[0]{\ensuremath{\mathcal{L}} }

\newcommand{\partderiv}[2]{\ensuremath{\frac{\partial #1}{\partial #2}} }

\newcommand{\half}[0]{\ensuremath{\frac{1}{2}} }

%-------------------------------------------------------------------------------

\newcommand{\ordset}[3] {\ensuremath{\left\{ {#1}_{#2} \right\}_{#2 = 1}^{#3}}}

%-------------------------------------------------------------------------------

\newcommand{\Eq}[1]  {Equation~(\ref{#1})}
\newcommand{\Eqs}[1] {Equations~(\ref{#1})}
\newcommand{\Sec}[1] {Section~\ref{#1}}
\newcommand{\Fig}[1] {Figure~\ref{#1}}
\newcommand{\Alg}[1] {Algorithm~\ref{#1}}
\newcommand{\Rem}[1] {Remark~\ref{#1}}
\newcommand{\Prop}[1] {Property~\ref{#1}}
\newcommand{\Thm}[1] {Theorem~\ref{#1}}
\newcommand{\Def}[1] {Definition~\ref{#1}}
\newcommand{\Tab}[1] {Table~\ref{#1}}

%-------------------------------------------------------------------------------

\newtheorem{theorem}{Theorem}[section]
\newtheorem{lemma}[theorem]{Lemma}
\newtheorem{proposition}[theorem]{Proposition}
\newtheorem{corollary}[theorem]{Corollary}
\newtheorem{conjecture}[theorem]{Conjecture}
   
\theoremstyle{definition}
\newtheorem{definition}[theorem]{Definition}
\newtheorem{example}[theorem]{Example}
\newtheorem{examples}[theorem]{Examples}

\theoremstyle{remark}
\newtheorem{remark}[theorem]{Remark}

%\numberwithin{equation}{section}

%-------------------------------------------------------------------------------